\documentclass[letterpaper,reqno]{amsart}

\usepackage[T1]{fontenc}
\usepackage[utf8]{inputenc}
\usepackage{amsthm,amsmath,amssymb,amsfonts,mathrsfs}
\usepackage{stmaryrd,braket}
\usepackage{hyperref}
\usepackage{setspace}

\theoremstyle{plain}
  \newtheorem{theorem}{Theorem}[section]
  \newtheorem{corollary}{Corollary}[section]
  \newtheorem{lemma}{Lemma}[section]
  \newtheorem{proposition}{Proposition}[section]
\theoremstyle{definition}
  \newtheorem{definition}{Definition}[section]
\theoremstyle{remark}
  \newtheorem{remark}{Remark}[section]

\newcommand{\C}{\mathbb{C}}
\renewcommand{\P}{\mathbb{P}}
\newcommand{\CP}{\mathbb{C}\mathrm{P}}

\newcommand{\tr}{\operatorname{tr}}
\newcommand{\V}{\mathcal{V}}
\newcommand{\F}{\mathcal{F}}

\newcommand{\BBindex}{\operatorname{BB}}
\newcommand{\CSindex}{\operatorname{CS}}
\newcommand{\Res}{\operatorname{Res}}
\newcommand{\Aff}[2]{\operatorname{Aff}({#1},{#2})}

\DeclareMathOperator{\Sing}{Sing}
\DeclareMathOperator{\Spec}{Spec}
\DeclareMathOperator{\LC}{LC}

\newcommand{\orcid}[1]{ORCID~ID:~\href{https://orcid.org/#1}{#1}}
\title[Spectra of~quadratic vector fields on~$\mathbb{C}^2$]{Spectra of~quadratic vector fields on~$\mathbb{C}^2$: the~missing relation}
\author[Yu.~Kudryashov]{Yury Kudryashov}
\address{Department of Mathematics, Cornell University\\ Ithaca, NY\\ 14850}
\email{\href{mailto:ik333@cornell.edu}{ik333@cornell.edu}}
\thanks{The first author is supported by the grant RFBR 16-01-00748.\\ \orcid{0000-0003-4286-9276}}
\author[V.~Ramírez]{Valente Ramírez}
\address{Institut de Recherche Math\'{e}matique de Rennes\\ Universit\'{e} de Rennes 1\\ UMR 6625\\ Rennes\\ France}
\email{\href{mailto:valente.ramirez@univ-rennes1.fr}{valente.ramirez@univ-rennes1.fr}}
\thanks{The second author was supported by the grants PAPIIT IN-106217 and CONACYT 219722. He also acknowledges the support of the Centre Henri Lebesgue ANR-11-LABX-0020-01.\\ \orcid{0000-0002-4594-430X}}

\hypersetup{
  pdftitle=\csname @title\endcsname,
  pdfauthor=Yury Kudryashov and Valente Ramírez,
  pdfkeywords=\csname @keywords\endcsname
}

\begin{document}

\begin{abstract}
    Consider a~quadratic vector field on $\C^2$ having an invariant line at~infinity and isolated singularities only.
    We define the \emph{extended spectra of~singularities} to be the collection of the spectra of the linearization matrices of each singular point over the affine part, together with all the characteristic numbers (i.e.~Camacho-Sad indices) at infinity.
    This collection consists of~$11$~complex numbers, and is invariant under affine equivalence of~vector fields.

    In this paper we describe all polynomial relations among these numbers.
    There are $5$ independent polynomial relations;
    four of~them follow from the Euler-Jacobi, the Baum-Bott, and the Camacho-Sad index theorems, and are well-known.
    The fifth relation was, until now, completely unknown.
    We provide an~explicit formula for~the missing 5th relation, discuss it's meaning and prove that it cannot be formulated as~an~index theorem.
\end{abstract}

\maketitle

\section{Introduction}

This work deals with \emph{generic quadratic vector fields} on the affine plane~$\C^2$, and the singular holomorphic foliations that these vector fields define on the projective plane $\P^2$.

The space of polynomial vector fields of degree at most $n$ has a natural vector space structure.
We will say that a property is \emph{generic} for vector fields of degree~$n$, if it is satisfied by~every vector field in~some dense Zariski open subset of~this vector space.
For example, a~generic vector field~$v$ of~degree~$n$ has exactly~$n^2$ isolated singularities.
Also, it is well known that the \emph{foliation} of~$\C^2$ defined~by~$v$ can~be extended to a~foliation of~$\P^2$ with isolated singularities.
This extension~$\F_v$ is unique, and generically $\F_v$ has an~invariant line at~infinity.
This means that the line $\mathcal{L}=\P^2\setminus\C^2$, once a finite number of singularities are removed, is a leaf of the foliation $\F_v$.
In the generic case, the number of singularities on $\mathcal{L}$ is exactly $n+1$.
We deal exclusively with vector fields with these properties.

\subsection{The extended spectra of singularities}

Let $p$ be an isolated singular point of some vector field $v = P(x,y)\frac{\partial}{\partial x} + Q(x,y)\frac{\partial}{\partial y}$, and consider the \emph{linearization matrix}
\[
    Dv(p) =
        \left.
        \begin{pmatrix}
            P'_x & P'_y \\
            Q'_x & Q'_y \\
        \end{pmatrix}
        \right|_{(x,y)=p.}
\]
Analytically equivalent vector fields have conjugate linearization matrices, hence the spectrum of the linearization matrix at each singular point is an analytic invariant.

\begin{definition}
    Let $p$ be a singular point of $v$.
    We define the \emph{spectrum} of $v$ at $p$ as the ordered pair $\Spec(v,p) = (\tr Dv(p),\det Dv(p))$.
    The \emph{finite spectra of~singularities} of~$v$ is the set (formally, the multiset)
    \[
        \Spec v = \set{\Spec(v,p)|v(p)=0}.
    \]
\end{definition}

In order to study the extended foliation in a neighborhood of the line at infinity we introduce the following change of coordinates: $z=\frac{1}{x}$, $w=\frac{y}{x}$.
A simple computation shows that, in these coordinates, a generic degree $n$ polynomial vector field induces a foliation given by an equation of the form
\begin{equation}
    \label{eq:near-infty-apart}
    \frac{dz}{dw} = z\sum_{j=1}^{n+1} \frac{\lambda_j}{w-w_j} + O(z^2).
\end{equation}
The line at infinity is given by $\mathcal{L}=\{z=0\}$, and the singular points on it correspond to the poles $w_j$.
The \emph{characteristic numbers at infinity} are defined to be the residues~$\lambda_j$, which are precisely the Camacho-Sad indices $\lambda_j=\CSindex(\F_v,\mathcal{L},w_j)$.

\begin{definition}
    \label{def:extendedSpectra}
    The \emph{extended spectra of singularities} of a polynomial vector field $v$ is the collection of the finite spectra of singularities, together with the characteristic numbers at infinity.
\end{definition}

\begin{remark}
    Since each number in the extended spectra is a local analytic invariant, affine equivalent vector fields have the same extended spectra.
\end{remark}

Note that even though we work with local invariants, we only consider the spectra as a collection of these invariants taken over all singularities. Therefore, the questions we deal with are of a global nature.

\subsection{Relations coming from index theorems}

Let us fix the following class of~quadratic vector fields.

\begin{definition}%
    \label{def:classesV2A2}
    Denote by $\V_2$ the space of all quadratic vector fields $v$ on $\C^2$ such that
    \begin{itemize}
      \item $v$ has exactly~$4$ isolated singularities;
      \item the extended foliation~$\F_v$ has an invariant line at infinity carrying exactly~$3$ singular points.
    \end{itemize}
    For technical reasons, we pass to~a~finite cover, and assume that both finite and infinite singularities of~$v$ are enumerated.
\end{definition}

Then the extended spectra of~a~vector field $v\in\V_2$ is an ordered tuple of~$11$~complex numbers:
$8$~coming from the finite spectra and $3$~characteristic numbers at infinity.
The object of~this paper is to give a~complete description of~\emph{all the algebraic relations} among these $11$~numbers.

These numbers are related by four classical \emph{index theorems}:
\begin{align}
  \sum_{v(p)=0}\frac{1}{\det Dv(p)} &= 0, \label{eq:EJ1} \tag{EJ1} \\
  \sum_{v(p)=0}\frac{\tr Dv(p) }{\det Dv(p)} &= 0, \label{eq:EJ2} \tag{EJ2}\\
  \sum_{p\in\Sing\F_v} \hspace{-8pt}\BBindex(\F_v,p) &= 16, \label{eq:BB} \tag{BB} \\
  \sum_{p\in \mathcal{L}\cap\Sing\F_v}\hspace{-12pt}\CSindex(\F_v,\mathcal{L},p) &= 1. \label{eq:CS} \tag{CS}
\end{align}
These are, respectively, the Euler-Jacobi relations, the Baum-Bott theorem and the Camacho-Sad theorem, see \autoref{sec:index-theorems} for details.

\subsection{The new “hidden” relation}
Let~us do a~simple dimension count.
On~the one hand, the space $\V_2$ has dimension~$12$, and the affine group $\Aff{2}{\C}$ has dimension~$6$.
Therefore, the quotient (in the sense of Geometric Invariant Theory) $\V_2\sslash\Aff{2}{\C}$ has dimension~6.
On the other hand, the extended spectra, which~is of~dimension~$11$, modulo the $4$~equations above is a space of dimension~$7$.
This gap in the dimensions implies that there must exist at least one more algebraic relation among these numbers.
This last \emph{hidden relation} was, until very recently, completely unknown.
In order to describe this relation, let us introduce some notation.

Let $v\in\V_2$ have singularities $p_1,\ldots,p_4$ on $\C^2$ and singular points at infinity $w_1$, $w_2$, $w_3$.
Put $\Spec(v,p_k)=(t_k,d_k)$, and let $\lambda_j$ be the characteristic number of $w_j$.
\begin{remark}
  The fact that all $7$~singular points are different implies that all four determinants $d_k$, and all characteristic numbers $\lambda_j$ are non-zero numbers.
\end{remark}
Let $\Lambda$ denote the product $\Lambda = \lambda_1\lambda_2\lambda_3$ and define $\mathcal R$ be the graded polynomial ring $\mathcal R=\C[t_1,\dotsc,t_4,d_1,\dotsc,d_4,\Lambda]$, where the generators $t_k$ are of degree 1, $d_k$ are of degree 2, and $\Lambda$ is of degree zero.

\begin{theorem}
    \label{thm:hidden-no-details}
    There exists a~polynomial~$H\in\mathcal R$ such that
    \begin{enumerate}
    \item\label{item:H-relation}
      the extended spectra of any quadratic vector field in $\V_2$ satisfies
      \begin{equation}
        \label{eq:hiddenRelation}
        H(t;d;\Lambda)=0;
      \end{equation}
    \item\label{item:H-independent}
      the above equality is independent from~\eqref{eq:EJ1}--\eqref{eq:CS};
    \item\label{item:H-full}
      if~$F\in\C[t;d;\lambda]$ is~another polynomial vanishing on the extended spectra of every vector field in $\V_2$, then $F=0$ follows from \eqref{eq:hiddenRelation}, together with the equalities \eqref{eq:EJ1}--\eqref{eq:CS} and inequalities $d_k\neq 0$, $k=1,\dotsc,4$, and~$\lambda_j\neq 0$, $j=1,2,3$.
    \end{enumerate}
\end{theorem}

The polynomial~$H$ in~\autoref{thm:hidden-no-details} is~not uniquely defined.
To get a~uniquely defined polynomial (up~to rescaling), we can eliminate $t_4$ and $d_4$ using \eqref{eq:EJ1} and \eqref{eq:EJ2}.
Let~$\mathcal S$ be the subring of~$\mathcal R$ consisting of~polynomials not depending on~$t_4$, $d_4$ or $\Lambda$.

\begin{theorem}%
    \label{thm:hidden-relation-deg14}
    There exist polynomials~$H_0, H_1, H_2\in\mathcal S$ homogeneous of~weighted degree~$14$ such that the polynomial~$H=H_2\Lambda^2+H_1\Lambda+H_0$ is irreducible in $\mathcal R$, and satisfies the assertions of~\autoref{thm:hidden-no-details}.
    The polynomial~$H$ is~uniquely defined up~to rescaling by a non-zero complex number.
\end{theorem}

The explicit expression for~$H$ (available at \cite{GitHubRepo}) was obtained using a~computer algebra system, see \autoref{sub:hiddenRelation} for details.
Unfortunately, the polynomial~$H$ has a very long expression: it consists of~$996$~monomials.

Another natural goal is~to~find a~polynomial~$H$ which is \emph{diagonal symmetric}, i.e.\ invariant under applying the same permutation to $t_k$'s and $d_k$'s.
Geometrically, such permutation corresponds to~reenumeration of~the singular points~$p_k$.
\begin{theorem}%
    \label{thm:hidden-relation-deg10}
    There exist diagonal symmetric polynomials~$\tilde H_0, \tilde H_1, \tilde H_2\in\mathcal S$ homogeneous of~weighted degree~$10$ such that for~a~generic quadratic vector field we~have
    \begin{equation}
        \label{eq:tH-ideal}
        (d_1d_2d_3)^2\tilde H_k=d_4^4H_k,
    \end{equation}
    where $H_k$ are the polynomials from \autoref{thm:hidden-relation-deg14}.
    The polynomial~$\tilde H=\tilde H_2\Lambda^2+\tilde H_1\Lambda+\tilde H_0$ satisfies the assertions of~\autoref{thm:hidden-no-details}.
\end{theorem}
The polynomial~$\tilde H$ \emph{is not} uniquely defined.
We provide an~explicit formula for one of~such polynomials in~\autoref{sec:explicit-formulas}.
\subsection{Lack of new index theorems}
Despite the length of the formula for~$H$, we have used its explicit expression to show that~\eqref{eq:hiddenRelation} does not come from an index theorem.
Moreover, we show that any possible “index-theorem-like identity” can be deduced from the four classical index theorems, hence concluding the lack of existence of new index theorems that constrain the extended spectra of quadratic vector fields.

The lack of existence of new index theorems is discussed in \autoref{sec:indexTheory}, but it follows from the next theorem.

\begin{theorem}
    \label{thm:noIndexTheorem}
    There exists no~pair~${(R,r)}$ consisting of~a~rational function~$R$ on~$\C^8$ and a~symmetric rational function~$r$ on~$\C^3$ with the property that every quadratic vector field with non-degenerate singularities satisfies the relation
    \[
        R(t;d)=r(\lambda),
    \]
    except for those that can be derived from the previously known relations \eqref{eq:EJ1}--\eqref{eq:CS}.
\end{theorem}

\begin{remark}
    The question of finding the hidden relations on the extended spectra was inspired by a very similar question on the hidden relations between the spectra of the derivatives at the fixed points of a rational endomorphism $f\colon\P^n\to\P^n$ posed by Adolfo Guillot in \cite{Guillot2004}.
    In fact, the case of vector fields may be understood as a sub-case of Guillot's problem \cite{WoodsHole}.
\end{remark}

\section{The classical index theorems}
\label{sec:index-theorems}

\subsection{The Euler-Jacobi relations}

Let us recall, in the particular case relevant to us, a classical result known as the Euler-Jacobi formula \cite[Chpt.~5, Sec.~2]{GriffithsHarris1994}.

\begin{theorem}
    \label{thm:EJformula}
    If $P,Q$ are polynomials in $\C[x,y]$ of degree $n$ whose divisors intersect transversely at $n^2$ different points $p_1,\ldots,p_{n^2}\in\C^2$ and $g(x,y)$ is a polynomial of degree at most $2n-3$, then
    \begin{equation}
        \label{eq:EJ-general}
        \sum_{k=1}^{n^2} \frac{g(p_k)}{\mathbf{J}(p_k)}=0,
    \end{equation}
    where $\mathbf{J}(x,y)$ is the Jacobian determinant $\mathbf{J}(x,y)=\det\displaystyle\frac{\partial(P,Q)}{\partial(x,y)}$.
\end{theorem}

Consider a polynomial vector field $v=P\frac{\partial}{\partial x}+Q\frac{\partial}{\partial y}$ of degree $n\geq2$.
Substituting $g(x,y)=1$ or $g(x,y)=\tr{Dv(x,y)}$ into \eqref{eq:EJ-general}, we obtain \eqref{eq:EJ1} or \eqref{eq:EJ2}, respectively.

\begin{corollary}
    \label{coro:EJrelations}
    A quadratic vector field $v$ having four non-degenerate singularities $p_1,\ldots,p_4\in\C^2$ satisfies \eqref{eq:EJ1} and \eqref{eq:EJ2}.
\end{corollary}
We call these equations the \emph{Euler-Jacobi relations on spectra}.

\begin{remark}
    In the end, the relations \eqref{eq:EJ1} and \eqref{eq:EJ2} come form the \emph{residue theorem} \cite[Chpt.~5, Sec.~1]{GriffithsHarris1994}: the local indices are nothing more than the residues of the rational 2-forms
    \[
        \frac{dx\wedge dy}{PQ}, \quad \text{and} \quad \frac{(\tr Dv)\,dx\wedge dy}{PQ}
    \]
    at the points $p_k$.
    The residue theorem then implies the total sum is zero.
\end{remark}
\subsection{The Baum-Bott theorem}

The Euler-Jacobi indices are well defined for singularities of vector fields, but not for foliations.
One of the most important invariants of an isolated singularity of a planar foliation is the \emph{Baum-Bott index}.
Suppose the germ of a foliation $(\F,p)$ with an isolated singularity is given by a holomorphic 1-form $\omega$.
The Baum-Bott index of $\F$ at $p$ is defined as
\[
    \BBindex(\F,p) = \frac{1}{(2\pi i)^2}\int_{\Gamma} \beta\wedge d\beta,
\]
where $\Gamma$ is the boundary of a small ball centered at $p$, and $\beta$ is~a~smooth $(1,0)$-form that satisfies $d\omega=\beta\wedge\omega$ in a neighborhood of $\Gamma$.
In the particular case where $\F$ is locally given by a non-degenerate vector field $v$, the index can be easily computed as
\[
    \BBindex(\F,p) = \frac{\tr^{\,2} Dv(p)}{\det Dv(p)}.
\]

The Baum-Bott theorem, originally proved in \cite{BaumBott1970} in a more general setting, can be stated in our particular case as follows \cite{Brunella2015}:

\begin{theorem}
    Let $\F$ be a foliation of projective degree $d$ on $\P^2$.
    Then
    \[
        \sum_{p\in\Sing\F} \hspace{-8pt}\BBindex(\F,p) = (d+2)^2.
    \]
\end{theorem}

\begin{corollary}
    Let $v\in\V_2$ have finite spectra $\{(t_k,d_k)\}$ and characteristic numbers at infinity $\{\lambda_j\}$.
    Then
    \begin{equation}
        \label{eq:BB-deg2}
        \sum_{k=1}^4\frac{t_k^2}{d_k} + \sum_{j=1}^3 \frac{(\lambda_j+1)^2}{\lambda_j} = 16.
    \end{equation}
\end{corollary}

\subsection{The Camacho-Sad theorem}

The Camacho-Sad theorem concerns singularities of a foliation along an invariant curve.
Suppose $C$ is a smooth curve invariant by a foliation $\F$ (see~\cite{Brunella2015} for the general case).
If $p$ is an isolated singularity of $\F$ on $C$, we can choose a local holomorphic 1-form $\omega$ generating $\F$ and a local equation $f$ for $C$ to obtain a decomposition
\[
    \omega = h\,df + f\eta,
\]
where $h$ is a holomorphic function and $\eta$ is a holomorphic 1-form.
In this case, the Camacho-Sad index is defined as follows:
\[
    \CSindex(\F,C,p) = -\frac{1}{2\pi i}\int_{\gamma} \frac{\eta}{h},
\]
where $\gamma\subset C$ is the boundary of a small disk centered at $p$.

\begin{theorem}
    [\cite{CamachoSad1982}]
    Let $\F$ be a foliation on a complex surface $S$ and let $C\subset S$ be a compact $\F$-invariant curve.
    Then
    \[
        \sum_{p\in C\cap\Sing\F}\hspace{-12pt}\CSindex(\F,C,p) = C\cdot C,
    \]
    where $C\cdot C$ denotes the self intersection number of $C$ in $S$.
\end{theorem}

Note that \eqref{eq:near-infty-apart} implies that $\CSindex(\F, \mathcal L, w_j)=\lambda_j$, so we have the following corollary.
\begin{corollary}
    \label{cor:sum-lambdas}
    The characteristic numbers at infinity of a vector field $v\in\V_2$ satisfy the relation
    \begin{equation}
        \label{eq:sum-lambdas}
        \lambda_1+\lambda_2+\lambda_3 = 1,
    \end{equation}
    which is clearly equivalent to~\eqref{eq:CS}.
\end{corollary}
This well-known relation may also be verified directly cf.~\autoref{rem:CS}.

\section{Twin vector fields}
\label{sec:twin-vfs}
In~a~previous paper of~the second author, the question of~whether or~not a~generic quadratic vector field (up to affine equivalence) is completely determined by its spectra (finite or extended) was studied.
The answer is that the finite spectra does not determine the vector field completely, there is a finite ambiguity coming from the existence of \emph{twin vector fields}.

\begin{definition}
    \label{def:twins}
    We will say that two vector fields $v$ and $v'$ are \emph{twin vector fields}, if they are not equal yet they have exactly the same singular locus and, for each point $p$ in the common singular set, the matrices $Dv(p)$ and $Dv'(p)$ have the same spectrum.
\end{definition}

\begin{theorem}
    [\cite{TwinVectorFields}]
    \label{thm:twins}
    A generic quadratic vector field has exactly one twin.
    Moreover, if two vector fields from the class $\V_2$ have the same finite spectra (no assumption on the position of the singularities) then, after transforming one of them by a suitable affine map, they are either identical or a~pair of twin vector fields.
\end{theorem}

The above theorem implies that given a generic vector field $v\in\V_2$, there exist exactly two disjoint orbits of the action of $\Aff{2}{\C}$ on $\V_2$ consisting of vector fields having the same finite spectra:
the orbit of $v$ and the orbit of its twin.
Note that this result is consistent with the dimensional count done before:
$\V_2\sslash\Aff{2}{\C}$ has dimension~$6$ and the space of~finite spectra, which consists of~$8$~complex numbers constrained by two Euler-Jacobi relations, has dimension~$6$ as~well.

A similar statement about the Baum-Bott index for foliations of~$\CP^2$ of \emph{projective} degree~$2$ was proved by Lins Neto: 
in~\cite{LinsNeto2012} it is proved that the generic fiber of the \emph{Baum-Bott map} consists of exactly $240$ orbits of the natural action of $\operatorname{Aut}(\CP^2)$ on the space of foliations.

\begin{remark}
    \label{rmk:twin-is-rational}
    Given a generic vector field $v\in\V_2$ we can compute its twin $v'$ by solving a simple system of algebraic equations (cf.~\cite{TwinVectorFields}). The coefficients defining $v'$ are expressed as rational functions on the coefficients defining $v$. Thus, we obtain a rational map (which is an involution) $\tau\colon\V_2\dasharrow\V_2$. This map is equivariant with respect to the action of the affine group on $\V_2$, and so descends to the quotient $\V_2\sslash\Aff{2}{\C}$ as a rational involution that we also denote by $\tau$.
\end{remark}

As the next theorem shows, twin vector fields have the same finite spectra but necessarily unequal characteristic numbers at infinity.

\begin{theorem}
    [\cite{TwinVectorFields}]
    \label{thm:moduliSpectra}
    Two generic quadratic vector fields are affine equivalent if and only if their extended spectra of singularities coincide.
\end{theorem}

\section{Predicting the form of the hidden relation}
\label{sec:predicting}

\subsection{Preliminaries}

We seek for an algebraic relation among the 11 numbers in the extended spectra. We shall work only with the following 7 variables: $t_1,t_2,t_3$, $d_1,d_2,d_3$, $\Lambda$ (recall that $\Lambda$ is the product of characteristic  numbers $\lambda_1\lambda_2\lambda_3$), from these we can recover the full extended spectra using the previously known equations \eqref{eq:EJ1}--\eqref{eq:CS}.

A first hope would be to explicitly write $\Lambda$ as a function of $t_1,\ldots,d_3$. However, this would be impossible: if this was the case then the finite spectra would completely determine the extended spectra and by \autoref{thm:moduliSpectra} it would completely determine the vector field (up to affine equivalence). This contradicts \autoref{thm:twins}. In this section we explain how we can achieve the next best thing. Namely, we can implicitly express $\Lambda$ in terms of $t_1,\ldots,d_3$ through a quadratic equation whose two roots correspond to the two values that $\Lambda$ may take (one for each of the two twins that realize the finite spectra $t_1,\ldots,d_3$).

If the coefficients of the above quadratic expression are polynomial, we obtain an equation $H_2\Lambda^2+H_1\Lambda+H_0=0$ as in \autoref{thm:hidden-relation-deg14}.

\subsection{Building the hidden relation}
\label{sub:builiding}
There are three maps that we need to build up the hidden relation, the fact that all three maps are rational is the key fact that allow us to recover a polynomial equation in the end. The maps are the following:

\begin{enumerate}
    \item The rational involution $\tau\colon\V_2\sslash\Aff{2}{\C}\dasharrow\V_2\sslash\Aff{2}{\C}$ in \autoref{rmk:twin-is-rational} that assigns to each vector field its unique twin,
    \item The rational map $\Lambda\colon\V_2\dasharrow\C$ that assigns to each vector field its product of characteristic numbers $\Lambda(v)=\lambda_1\lambda_2\lambda_3$ (see \autoref{sec:product-infty} for a detailed discussion on this map),
    \item The birational isomorphism $\psi\colon\C^6\dasharrow\V_2\sslash\Aff{2}{\C}\sslash\tau$ defined as the birational inverse of the map $[v]\mapsto(t_1,\ldots,d_3)$.
\end{enumerate}

We remark that the map $\V_2\dasharrow\C^6\sslash\Aff{2}{\C}$ which maps $v$ to $(t_1,\ldots,d_3)$ is a dominant rational map whose generic fiber is two points corresponding to a pair of twins (cf.~\cite{TwinVectorFields}), thus once we pass to the quotient by the involution $\tau$ we obtain a birational isomorphism.

We're now ready to construct the hidden relation. The map $\V_2\sslash\Aff{2}{\C}\dasharrow\C$ given by $v\mapsto \Lambda(v)+\Lambda(\tau(v))$ is invariant with respect to $\tau$ and so descends to the quotient $\V_2\sslash\Aff{2}{\C}\sslash\tau$. Precomposing with $\psi$ we obtain a map

\begin{align*}
    R_1\colon\C^6 & \dasharrow \C \\
    (t_1,\ldots,d_3) &\longmapsto \Lambda(v)+\Lambda(\tau(v))
\end{align*}

In a similar way we construct a map $R_2(t_1,\ldots,d_3)=\Lambda(v)\Lambda(\tau(v))$. It follows now that the spectra of a generic quadratic vector field satisfies the equation

\[
    \Lambda^2 - R_1(t;d)\Lambda + R_2(t;d) = 0.
\]

From this discussion we deduce the existence of a new relation, and we see why it ought to be a quadratic equation on $\Lambda$. An algorithm to obtain the hidden relation following the above arguments is described in \autoref{subsub:slow}. However, in order to compute the hidden relation in a more efficient way, we're going to follow a different strategy (\autoref{subsub:fast}).

\section{Computing the hidden relation}
\label{sec:relation}
\subsection{Plan of the proof}\label{sub:plan}
In this section we are going to prove in detail Theorems~\ref{thm:hidden-no-details}--\ref{thm:hidden-relation-deg10}. 
The heuristic idea behind the proof diverges from \autoref{sec:predicting} but is straightforward.
Consider the map $\mathcal{M}\colon\V_2\sslash\Aff{2}{\C}\to\C^7$ that assigns to each vector field $v$ the tuple $(t_1, t_2, t_3, d_1, d_2, d_3, \Lambda)$, where $\Lambda=\lambda_1\lambda_2\lambda_3$.
In terms of \autoref{sec:predicting}, $\mathcal M=(\psi^{-1}, \Lambda)$.
Our goal is to describe the image of $\mathcal M$.
It turns out that $\mathcal M$ is~a~rational map, and we shall find an explicit formula for this map.
Finally, (the closure of) the image of $\mathcal M$ is the projection of the graph of~$\mathcal M$ to the codomain, so we can use standard Gröbner basis algorithms to find the vanishing ideal of this projection.

In \autoref{sec:product-infty} we will provide a useful way to express~$\Lambda$ in terms of the \emph{coefficients} of~$P$ and~$Q$.
Then in~\autoref{sub:normal-form} we provide explicit coordinates on $\V_2\sslash\Aff{2}{\C}$.
The explicit formulas for the finite spectra in these coordinates together with the relation from \autoref{sec:product-infty} give an explicit formula for~$\mathcal M$.
Finally, in~\autoref{sub:hiddenRelation} we will use this formula to~establish a~new relation between the \emph{finite} spectra, and the number~$\Lambda$.

In~\autoref{sub:proof-deg14} we will prove that the new relation satisfies the assertions of \autoref{thm:hidden-relation-deg14}, and in~\autoref{sub:proof-deg10} we will use it to prove \autoref{thm:hidden-relation-deg10}.

\subsection{The product of the characteristic numbers at infinity}
\label{sec:product-infty}
Consider a polynomial vector field $P(x, y)\frac{\partial}{\partial x}+Q(x, y)\frac{\partial}{\partial y}$ of degree $n$.
In the coordinates $z=\frac 1x$, $w=\frac yx$, it takes the form
\[
    \frac{dz}{dw}=z\frac{\tilde P(z, w)}{w\tilde P(z, w) - \tilde Q(z, w)},
\]
where $\tilde P(z, w)=z^nP\left( \frac 1z, \frac wz \right)$, $\tilde Q(z, w)=z^nQ\left( \frac 1z, \frac wz \right)$.
In the generic case $w\tilde P(0, w)-\tilde Q(0, w)$ is a non-zero polynomial of degree $n+1$, therefore
\begin{equation}
    \label{eq:near-infty-sr}
    \frac{dz}{dw}=z\frac{F(w)}{G(w)}+O(z^2),
\end{equation}
where $F(w)=\tilde P(0, w)$ and $G(w)=w\tilde P(0, w)-\tilde Q(0, w)$.
\begin{remark}
    \label{rem:CS}
    Note that the leading coefficients of~$F$ and~$G$ are both equal to the leading coefficient of~$\tilde P(0, w)$, hence $\lim_{w\to\infty}w\frac{F(w)}{G(w)}=1$.
    Comparing this equality and \eqref{eq:near-infty-sr} to \eqref{eq:near-infty-apart}, we get another proof of~\eqref{eq:CS}.
\end{remark}

Comparing \eqref{eq:near-infty-sr} to \eqref{eq:near-infty-apart}, we see that the characteristic numbers $\lambda_k$ are the residues of the rational function $\frac{F(w)}{G(w)}$ at the zeroes $w_k$ of its denominator, hence
\begin{equation}
    \label{eq:formula-lambdas}
    \lambda_k=\frac{F(w_k)}{G'(w_k)}.
\end{equation}

The following lemma allows us to find the product $\Lambda=\prod_{k=1}^{n+1}\lambda_k$ as a rational function of the coefficients of~$P$ and~$Q$.
\begin{lemma}
    \label{lemma:Lambda-resultants}
    Let $\Lambda$, $F$, $G$ be as above.
    Then
    \begin{equation}
        \label{eq:Lambda-resultants}
        \Lambda=\frac{\Res(F, G)}{\Res(G', G)}.
    \end{equation}
    Here and below $\Res(\cdot, \cdot)$ stands for the resultant of its arguments.
\end{lemma}
Recall that the resultant of two polynomials is a polynomial of their coefficients, hence the right hand side of~\eqref{eq:Lambda-resultants} is a rational function of the coefficients of~$P$ and~$Q$.
\begin{proof}
    Multiplying \eqref{eq:formula-lambdas} for all $k=1,2,\ldots n+1$ we obtain the formula
    \[
        \Lambda = \frac{\prod_{k=1}^{n+1} F(w_k)}{\prod_{k=1}^{n+1} G'(w_k)}.
    \]
    The numerator in the above expression is the product of $F(w)$ at each of the roots of $G$.
    This product is equal to
    \[
        \prod_{k=1}^{n+1} F(w_k) = \LC(G)^{-n}\Res(F, G),
    \]
    where $\LC(G)$ is the leading coefficient of~$G$.
    Similarly, the denominator equals $\LC(G)^{-n}\Res(G',G)$ and so we obtain~\eqref{eq:Lambda-resultants}.
\end{proof}

\subsection{Normal form}%
\label{sub:normal-form}
Note that a quadratic vector field from the class $\V_2$ cannot have three collinear singularities.
Indeed, otherwise it would vanish on the line passing through these singularities.
Hence we have the following lemma.

\begin{lemma}%
    \label{lemma:normalizeSingularities}
    Every quadratic vector field with four isolated singularities is affine equivalent to a vector field with singularities at $p_1=(0,0)$, $p_2=(1,0)$, $p_3=(0,1)$.
    Any such vector field $v=P\frac{\partial}{\partial x}+Q\frac{\partial}{\partial y}$ is defined by polynomials
    \begin{equation}
        \label{eq:coefficients}
        \begin{aligned}
            P(x,y) &= a_0x^2+a_1xy+a_2y^2-a_0x-a_2y,  \\
            Q(x,y) &= a_3x^2+a_4xy+a_5y^2-a_3x-a_5y,
        \end{aligned}
    \end{equation}
    for some complex numbers $a_0,\ldots,a_5$.
\end{lemma}

Under this “normal form” we can immediately compute the explicit expressions for the traces $t_k$ and determinants $d_k$ at each singular point $p_k$, for $k=1,2,3$.
Also, \autoref{lemma:Lambda-resultants} provides~us an~expression for~$\Lambda=\lambda_1\lambda_2\lambda_3$ as a~rational function of~$a_j$.
Explicit expressions for all these values are provided in~\autoref{lemma:expressionsSpectra} in \autoref{sec:explicit-formulas}.

\subsection{The hidden relation}%
\label{sub:hiddenRelation}
We have $7$~polynomial relations on $13$~variables $t_k$, $d_k$, $k=1,2,3$, $a_j$, $j=0,\dotsc,5$, $\Lambda$.
The common zero locus of these relations is the graph of the map $\mathcal M$ introduced in \autoref{sub:plan}.
Our goal is to eliminate $a_j$ from these equations.
Geometrically, this corresponds to projecting the graph of $\mathcal M$ to the codomain of $\mathcal M$, thus finding (the closure of) its image.

There are (at least) two ways to achieve this goal.
\subsubsection{Fast computation}
\label{subsub:fast}
Consider the ideal~$J$ generated by our relations, and use a computer algebra system to eliminate $a_j$ from this ideal.
It turns out that the resulting ideal~$I$ is generated by a~single polynomial~$H\in\C[t;d;\Lambda]$.
This polynomial will be the “hidden” relation.
We conclude that (the closure of) the image of $\mathcal M$ is a hypersurface in $\C^7$.
This agrees with the dimensional count: we had $7$ equations, then we eliminated $6$ variables, so we have $7-6=1$ equation left.

The polynomial~$H$ has~degree~$2$ in~$\Lambda$, and has~weighted degree~$14$, cf. \autoref{thm:hidden-relation-deg14}.
We have done this computation in \texttt{CoCoA~5}~\cite{CoCoA-5} and \texttt{Macaulay2}~\cite{Macaulay2}, getting the same polynomial~$H$ in both cases.
These computations are available at \cite{GitHubRepo}.

\subsubsection{Slow computation}
\label{subsub:slow}
Before implementing the fast algorithm described above, we have obtained the hidden relation following a different method which is closer to the ideas presented in \autoref{sec:predicting}.

The first step is to invert the map that sends a vector field $v$ to its finite spectra $(t;d)$. Consider the explicit formulas for $t_k$, $d_k$, see \autoref{sec:explicit-formulas}, as equations on $a_j$.
The formulas for $t_k$ are linear in $a_j$, and the formulas for $d_k$ have degree~$2$.
As in~\cite{TwinVectorFields}, \autoref{thm:EJformula} implies that the fourth zero of a vector field given~by~\eqref{eq:coefficients} is the point~$\left(-\frac{d_4}{d_2}, -\frac{d_4}{d_3}\right)$.
This fact adds two more linear equations to our system, and we need to solve a quadratic equation on the line in $\C^6$ given by $5$ linear equations.
Let $D\in\C[t;d]$ be the discriminant of this quadratic equation.
Then~$a_j$ are rational functions of $t_k$, $d_k$, and $\sqrt{D}$.
Then we substitute these expressions into~\eqref{eq:Lambda-resultants}, and get a formula for $\Lambda$ as a rational function of $t_k$, $d_k$, and $\sqrt{D}$.
This formula can be easily transformed into a polynomial equation $H=0$ quadratic in~$\Lambda$ (cf.~\autoref{sub:builiding}).
We have implemented this approach in \texttt{GiNaC}~\cite{GiNaC}, and the result agrees with the results of the “fast” method (cf.~\cite{GitHubRepo}).

\subsection{Proof of \autoref{thm:hidden-relation-deg14}}%
\label{sub:proof-deg14}
Let us prove that the polynomial~$H$ constructed above satisfies all assertions of~\autoref{thm:hidden-relation-deg14}.

The facts that~$H$ has weighted degree~$14$ and is irreducible were verified both in \texttt{CoCoA 5} and \texttt{Macaulay2}. By construction, the spectra of a quadratic vector field in $\V_2$ satisfies $H$, so assertion~\ref{item:H-relation} of \autoref{thm:hidden-no-details} is clear.

Let us prove that~$H$ satisfies assertion~\ref{item:H-independent}.
\begin{proposition}
    \label{prop:independence}
    The identity~$H=0$ is independent from the identities~\eqref{eq:EJ1}--\eqref{eq:CS} coming from the classical index theorems.
\end{proposition}
\begin{proof}
    Denote by $\sigma_j(\lambda)$ the $j$-th elementary symmetric polynomial on $\lambda=(\lambda_1,\lambda_2,\lambda_3)$.
    Then~\eqref{eq:CS} takes the form $\sigma_1(\lambda)=1$, and~\eqref{eq:BB} takes the form
    \begin{equation}
        \label{eq:sigma2BySigma3}
        \frac{\sigma_2(\lambda)}{\sigma_3(\lambda)} = -\sum_{k=1}^4 \frac{t_k^2}{d_k} + 9,
    \end{equation}
    cf \eqref{eq:BB-deg2}.
    Finally, $\sigma_3(\lambda)$ is the $\Lambda$ from the hidden relation \eqref{eq:hiddenRelation}.

    Note that we can consider both classical relations, and \eqref{eq:hiddenRelation} as relations on $t_k$, $d_k$, $k=1,\dotsc,4$, and $\sigma_j$, $j=1,2,3$.
    After the birational coordinate change
    \begin{equation}
        \label{eq:cs}
        (t;d;\sigma_1, \sigma_2, \sigma_3)\mapsto \left(t;d;\sigma_1, \frac{\sigma_2}{\sigma_3}, \sigma_3\right),
    \end{equation}
    the hidden relation~\eqref{eq:hiddenRelation} is the only relation that includes $\sigma_3=\Lambda$.
    Therefore, it is independent from the classical identities.
\end{proof}

Finally, let us show that~$H$ satisfies assertion~\ref{item:H-full}.
Consider a~polynomial~$F\in\C[t;d;\lambda]$ vanishing at the extended spectrum of~every quadratic vector field in $\V_2$.
It is easy to see that for each vector field $v\in\V_2$ with enumerated singularities, there exists a~path in~$\V_2$ joining~$v$ to any other element of~$\V_2$ that corresponds to~the same vector field but with another prescribed order of~singularities at~infinity (and the same order for~finite zeros).
Hence, without loss of~generality, we may assume that $F$ is symmetric in~$\lambda$, and so it can~be rewritten as~$F(t;d;\lambda)=\tilde F(t;d;\sigma(\lambda))$.
Using the classical relations, we can get rid of $t_4$, $d_4$, $\sigma_1$, $\sigma_2$ to obtain a rational expression on $t_1,\ldots,d_3,\Lambda$. The inequalities $d_k\neq0$, $\lambda_j\neq0$ allow us to lift denominators and recover a polynomial expression.

Therefore, in order to check assertion~\ref{item:H-full}, it is enough to consider polynomials $F\in\C[t_1,\dotsc,d_3,\Lambda]$.
If such polynomial~$F$ vanishes on all extended spectra of~vector fields $v\in\V_2$, then it belongs to the vanishing ideal~$I$ of (the closure of) the image of~$\mathcal M$.
Due to \autoref{subsub:fast}, $I$ is the principal ideal generated by~$H$.
Since $H$ is irreducible, it follows that $F$ is either identically zero or a multiple of $H$.
Hence, $F=0$ follows from $H=0$ and thus $H$ satisfies assertion~\ref{item:H-full}.

Finally, $H$ satisfies all the assertions of \autoref{thm:hidden-no-details}.
\subsection{Reducing the degree of~$H$}%
\label{sub:proof-deg10}
Consider the polynomial~$H$ as~a~function of~the spectra of~all four finite zeros.
It turns out that this function is \emph{not} symmetric under permutations of the zeros.
More precisely, if we substitute into~$H$ formulas for $t_4$, $d_4$ instead of, e.g., $t_3$, $d_3$, the resulting function will be proportional to~$H$, but not equal to it.

Among components of~$H$, the “free term”~$H_0$ has the simplest formula.
Looking at this formula, we guessed that the function~${(d_1d_2d_3)}^{-2}d_4^4H$ is diagonal symmetric, i.e., it survives under permutations of the finite zeros of~$v\in\V_2$.
Moreover, this function is equal to a polynomial~$\tilde H\in\C[t_1,\dotsc,t_4;d_1,\dotsc,d_4]$ on the submanifold given by~\eqref{eq:EJ1} and~\eqref{eq:EJ2}.

Formally, one can check that $d_4^4H$ belongs to the ideal generated by the numerators of~\eqref{eq:EJ1},~\eqref{eq:EJ2}, and ${(d_1d_2d_3)}^2$, hence $d_4^4H-{(d_1d_2d_3)}^2\tilde H$ belongs to the ideal~$EJ$ generated by the numerators of~\eqref{eq:EJ1},~\eqref{eq:EJ2} for some polynomial~$\tilde H$, and we can find an~explicit formula for~$\tilde H$.
It is easy to see that this polynomial~$\tilde H$ satisfies the assertions of~\autoref{thm:hidden-relation-deg10}.

The polynomial~$\tilde H$ \emph{is not} uniquely defined.
Indeed, we can replace it by any other polynomial $\tilde H'$ such that ${(d_1d_2d_3)}^2(\tilde H-\tilde H')$ belongs to the ideal~$EJ$.
All the polynomials~$\tilde H$ we have found have more than $1{,}500$ terms, but some of them are diagonal symmetrizations of relatively short polynomials.
The explicit formula for one of these polynomials is given in \autoref{sec:explicit-formulas}.
\section{Index theory}\label{sec:indexTheory}
In this section we prove \autoref{thm:noIndexTheorem}.
In \autoref{sub:ind-motivation} we explain how this theorem relates to the index theory.
Next, in \autoref{sub:ind-proof} we prove this theorem modulo the main lemma whose proof is postponed till \autoref{sub:ind-key-lemma}.
\subsection{Motivation for \autoref{thm:noIndexTheorem}}\label{sub:ind-motivation}
The study of indices and index theorems is fundamental in the development of geometry and topology.
Theorems like the Poincaré-Hopf index theorem, the Gauss-Bonnet theorem or the Lefschetz fixed-point theorem are a few important examples.
These \emph{local-to-global} theorems, which relate the local behavior of some geometric object around “special” points to some global invariant (usually measured in some cohomology space) are indeed powerful and fascinating.

The theorems we previously knew that relate the extended spectra of a polynomial vector field are of all this type, where a sum of local contributions (an index) taken over all singular points equals a fixed number that depends only on the degree of the vector field.
The indices in question have been defined as \emph{residues} of meromorphic forms, and can also be understood as \emph{localizations} of characteristic classes (see for example \cite{Suwa1998}).

It was proved in~\cite{WoodsHole} that all four classical relations can be realized as particular cases of the so-called \emph{Woods Hole trace formula}, also known as the \emph{Atiyah-Bott fixed point theorem} (a generalization of the Lefschetz fixed point theorem due to Atiyah and Bott in the complex analytic case, and to Verdier in the algebraic case), and one could be tempted to think that the hidden relation would also be of this type.
This is not the case.
As it will be shown in this section, there are no more index theorems that relate the extended spectra other than those which can be derived from the classical ones. In particular, the hidden relation does not come from an index theorem.

\begin{remark}
    For quadratic vector fields we had only one hidden relation.
    However, as the degree $n$ of the vector fields grows, the number of hidden relations grows asymptotically as $n^2$.
    Indeed, on one hand a generic polynomial vector field of degree $n$ has exactly $n^2$ zeroes and $n+1$ singular points at infinity.
    Thus, the extended spectra consists of $2n^2+n+1$ complex numbers.
    On the other hand, the space of degree $n$ polynomial vector fields is of dimension $n^2+3n+2$.
    It is hard to imagine that there would be a countable number of index theorems to account for all these hidden relations.
    Therefore, in this way, it is not surprising that the hidden relations do not come from index theorems.
\end{remark}

Below we formalize the statement that the hidden relation for quadratic fields does not come from an index theorem.

An index ought to be a number we assign to an isolated singularity of a vector field, foliation, space or map.
This index should only depend on the local behavior of our geometric object around such singular point.
Moreover, this index should be invariant under analytic equivalence.
We know from Poincaré that a \emph{complex hyperbolic} singularity of a planar vector field is analytically linearizable.
Analytic invariance of the index implies that the index cannot depend on anything but the spectra.

\begin{remark}
    \label{rmk:expectedIndex}
    An index theorem for generic polynomial vector fields of degree~$n$ on~$\C^2$ should be such that the local-to-global equation is of the following form:
    \[
        \sum_{k=1}^{n^2} \operatorname{ind}_{\C^2}(v,p_k) + \sum_{j=1}^{n+1} \operatorname{ind}_{\mathcal{L}}(\F_v,w_j) = L(n),
    \]
    where $\operatorname{ind}_{\C^2}(v,p_k) $ is a rational function on the spectrum of the linearization matrix $Dv(p_k)$, $\operatorname{ind}_{\mathcal{L}}(\F_v,w_j)$ is a rational function on the characteristic number $\lambda_k$ of the singularity at infinity $w_k$, and $L(n)$ is a number that depends only on the degree $n$ of the vector field.
\end{remark}

Note that in particular the above equation may be rewritten as
\begin{equation}
    \label{eq:R=r}
    R(t;d) = r(\lambda),
\end{equation}
where $R$ is a rational function on the finite spectra $\{(t_k,d_k)\}$, and $r$ is a rational function on the characteristic numbers at infinity $\{\lambda_j\}$, and they are invariant under permutations of~finite or~infinite singularities.
Moreover, all classical relations \eqref{eq:EJ1}--\eqref{eq:CS} are of this form.
Recall that \autoref{thm:noIndexTheorem}, which we will prove below, states that \emph{any} relation of the form \eqref{eq:R=r} follows from the classical relations.

\subsection{Lack of new index theorems}
\label{sub:ind-proof}
In order to prove \autoref{thm:noIndexTheorem} we will first use the classical relations to eliminate some variables.
Since $r$ is symmetric, we can rewrite it in terms of $\sigma_1(\lambda)$, $\frac{\sigma_2(\lambda)}{\sigma_3(\lambda)}$ and $\Lambda=\sigma_3(\lambda)$, cf.~\eqref{eq:cs}, then substitute $\sigma_1(\lambda)=1$, see \eqref{eq:CS}:
\[
    r(\lambda)=\tilde r\left(\frac{\sigma_2(\lambda)}{\sigma_3(\lambda)}, \sigma_3(\lambda)\right).
\]

Note that for twin vector fields $v$, $v'$ with characteristic numbers $\lambda$, $\lambda'$ we have $r(\lambda)=r(\lambda')$, hence
\begin{equation}
    \label{eq:tilde-r-r'}
    \tilde r\left(\frac{\sigma_2(\lambda)}{\sigma_3(\lambda)}, \sigma_3(\lambda)\right)
    =
    \tilde r\left(\frac{\sigma_2(\lambda')}{\sigma_3(\lambda')}, \sigma_3(\lambda')\right).
\end{equation}
Twin vector fields have the same finite spectra, hence \eqref{eq:sigma2BySigma3} implies that $\frac{\sigma_2(\lambda)}{\sigma_3(\lambda)}=\frac{\sigma_2(\lambda')}{\sigma_3(\lambda')}$.
Therefore, \eqref{eq:tilde-r-r'} holds whenever~$\tilde r$ does not depend on its second argument.
The following lemma states that this is the only possibility.
\begin{lemma}
    \label{lemma:r-s2-s3}
    Let $\tilde r$ be a rational function such that for twin vector fields $v$, $v'$ with characteristic numbers at infinity $\lambda$, $\lambda'$ we have \eqref{eq:tilde-r-r'}.
    Then $\tilde r$ depends only on its first argument, $\tilde r(\xi, \chi)=\hat r(\xi)$.
\end{lemma}
We shall prove this lemma in \autoref{sub:ind-key-lemma}.
Now we apply it to $\tilde r$, and rewrite \eqref{eq:R=r} in the form
\[
    R(t;d)=\hat r\left(\frac{\sigma_2(\lambda)}{\sigma_3(\lambda)}\right).
\]
Substituting \eqref{eq:sigma2BySigma3}, we get
\[
    R(t;d)=\hat r\left(-\sum_{k=1}^4\frac{t_k^2}{d_k}+9\right).
\]
Note that this relation \emph{does not} involve $\lambda_j$.
Finally, we use \eqref{eq:EJ1} and \eqref{eq:EJ2} to get rid of $t_4$ and $d_4$, and get a~relation of~the~form $\tilde R(t_1, t_2, t_3, d_1, d_2, d_3)=0$.
Now we have applied all the relations \eqref{eq:EJ1}--\eqref{eq:CS}, so we need to show that $\tilde R\equiv 0$.

It was proved in \cite[Theorem 1']{TwinVectorFields} that the map that takes each vector field in the normal form described in \autoref{lemma:normalizeSingularities} to the set $(t_1,d_1,\ldots,t_3,d_3)\in\C^6$ is a dominant map (i.e.~the image is dense in the Zariski topology). This implies that there are no non-trivial polynomials that vanish on every tuple $(t_1,d_1,\ldots,t_3,d_3)$.
In particular, $\tilde{R}=0$ for every generic vector field $v$ implies that $\tilde{R}\equiv 0$.

This completes the proof of \autoref{thm:noIndexTheorem}, modulo the main \autoref{lemma:r-s2-s3}.

\subsection{Proof of the main \autoref{lemma:r-s2-s3}}
\label{sub:ind-key-lemma}

In order to prove \autoref{lemma:r-s2-s3}, for each pair of twin vector fields $v$, $v'$, consider the triple $\left(\frac{\sigma_2(\lambda)}{\sigma_3(\lambda)}, \sigma_3(\lambda), \sigma_3(\lambda')\right)$.
Denote by $X\subset\C^3$ the set of these triples for all pairs of twin vector fields.
Then for $(\xi,\Lambda,\Lambda')\in X$ we have $\tilde r(\xi, \Lambda)=\tilde r(\xi, \Lambda')$.
Now, it is enough to show that $X$ has an inner point in $\C^3$.
Indeed, this would imply that~$X$ is Zariski dense in~$\C^3$, thus $\tilde r(\xi, \Lambda)=\tilde r(\xi, \Lambda')$ for \emph{all} $(\xi, \Lambda, \Lambda')\in\C^3$, hence $\tilde r$ does not depend on its second argument.

Note that $\Lambda=\sigma_3(\lambda)$ and $\Lambda'=\sigma_3(\lambda')$ are the roots of~\eqref{eq:hiddenRelation}.
Hence, Vieta's formulas imply
\begin{align*}
    \frac{H_1}{H_2}&=-\Lambda-\Lambda'\\
    \frac{H_0}{H_2}&=\Lambda\Lambda'.
\end{align*}
Consider the rational map $\Phi\colon\C^8\dashrightarrow\C^3$ given by
\begin{align*}
    \Phi(t;d)
        &=\left(-\sum_{k=1}^4\frac{t_k^2}{d_k}+9, -\frac{H_1}{H_2}, \frac{H_0}{H_2}\right)\\
        &=\left(\frac{\sigma_2(\lambda)}{\sigma_3(\lambda)}, \Lambda+\Lambda', \Lambda\Lambda'\right).
\end{align*}
Recall that we know explicit formula for~$H$, hence we have an explicit formula for~$\Phi$.
Let $\tilde\Phi$ be the restriction of~$\Phi$ to the affine submanifold given by \eqref{eq:EJ1}, \eqref{eq:EJ2}.
We computed the rank of $\tilde\Phi$ at the point $(1, 1, 2; 1, 2, -1)$ in \texttt{CoCoA}.
This rank equals~$3$, hence $\tilde\Phi(1, 1, 2; 1, 2, -1)$ is~an~inner point of~the~image of~$\tilde\Phi$.

Since \eqref{eq:EJ1}, \eqref{eq:EJ2} are the only relations on the finite spectra, the set
\[
    \tilde X=\set{(\xi, \Lambda+\Lambda', \Lambda\Lambda')|(\xi, \Lambda, \Lambda')\in X}\subset \operatorname{im} \tilde\Phi
\]
has an inner point as well.
Therefore, $X$ has an inner point.
This completes the proof of \autoref{lemma:r-s2-s3}, hence the proof of \autoref{thm:noIndexTheorem}.

\section*{Acknowledgements}

An early version of this work appeared as part of the doctoral dissertation of the second author, who would like to thank Yulij Ilyashenko, John Hubbard, and John Guckenheimer for their guidance throughout his Ph.D.
We are also grateful to Dominique Cerveau, Étienne Ghys, Adolfo Guillot, and Frank Loray for fruitful conversations on this project.
We would also like to thank Nataliya Goncharuk for proofreading the preliminary versions of the article.

\appendix
\section{Explicit formulas}
\label{sec:explicit-formulas}

In this appendix we include the explicit expressions for the spectra and the product of characteristic numbers~$\Lambda$ announced in \autoref{sub:normal-form}.

As in \autoref{lemma:normalizeSingularities}, put $p_1=(0,0)$, $p_2=(1,0)$, $p_3=(0,1)$.
Consider a~quadratic vector field~$v$ having isolated singularities at $p_k$, $k=1,2,3$.
This vector field is given~by~\eqref{eq:coefficients}.
\begin{lemma}
    \label{lemma:expressionsSpectra}
    The spectra $(t_k,d_k)$ of~$v$ at~$p_k$, $k=1,2,3$, is given by
    \begin{align*}
        t_1 &= -a_0-a_5;         & d_1 &= -a_2a_3+a_0a_5; \\
        t_2 &= a_0+a_4-a_5;      & d_2 &= -a_1a_3+a_2a_3+a_0a_4-a_0a_5; \\
        t_3 &= -a_0+a_1+a_5;     & d_3 &= a_2a_3-a_2a_4-a_0a_5+a_1a_5.
    \end{align*}
    The numerator and the denominator in~\eqref{eq:Lambda-resultants} are given~by
    \begin{align*}
        a_2^{-1}\Res(F, G)&=-(d_1d_2+d_2d_3+d_1d_3)=\frac{d_1d_2d_3}{d_4},\\
        a_2^{-1}\Res(G', G)
          &= 4a_0^3a_2 - a_0^2a_1^2 + 2a_0^2a_1a_5 - 12a_0^2a_2a_4 -
             a_0^2a_5^2 + 2a_0a_1^2a_4 + \\
          & \mathrel{\phantom{=}}
             18a_0a_1a_2a_3 - 4a_0a_1a_4a_5 - 18a_0a_2a_3a_5 + 12a_0a_2a_4^2 +
             2a_0a_4a_5^2 - \\
          & \mathrel{\phantom{=}}
             4a_1^3a_3 + 12a_1^2a_3a_5 - a_1^2a_4^2 - 18a_1a_2a_3a_4 -
             12a_1a_3a_5^2 + 2a_1a_4^2a_5 + \\
          & \mathrel{\phantom{=}}
             27a_2^2a_3^2 + 18a_2a_3a_4a_5 - 4a_2a_4^3 + 4a_3a_5^3 - a_4^2a_5^2.
    \end{align*}
    Their ratio is equal to $\Lambda=\lambda_1\lambda_2\lambda_3$.
\end{lemma}
The second equality in the expression for~$\Res(F, G)$ follows from~\eqref{eq:EJ1}, and the last conclusion follows from \autoref{lemma:Lambda-resultants}.
The rest of the lemma can be proved by a~direct computation.
Though the formula for $\Res(F, G)$ is much shorter than the formula for $\Res(G', G)$, we have no geometric interpretation for neither of these formulas.

Finally, we provide explicit formulas for the polynomials~$\tilde H_j$ from \autoref{thm:hidden-relation-deg10}.
In order to save space, we write the formula for polynomials~$\hat H_j$ such that $\tilde H_j$ are their symmetrizations with respect to~the diagonal action of~$S_4$,
\[
    \tilde H_j = \sum_{\sigma\in S_4} \hat H_j(t_{\sigma(1)}, t_{\sigma(2)}, t_{\sigma(3)}, t_{\sigma(4)}, d_{\sigma(1)}, d_{\sigma(2)}, d_{\sigma(3)}, d_{\sigma(4)}).
\]

Here are the explicit formulas.
\bgroup\allowdisplaybreaks
\begin{align*}
  \hat H_0
      &=-4 d_1^3d_2d_3 + 4d_1^2d_2^2d_3\\
  \hat H_1
      &=2 d_2^2d_3t_1^3t_2 + 2d_3^2d_4t_1^3t_2 - 6d_1d_3^2t_1^2t_2^2 -
        2d_3^2d_4t_1^2t_2^2 -2d_2^2d_4t_1^2t_2t_3 +\\
      &\phantom{{}={}}
        6d_1d_4^2t_1^2t_2t_3 + 36d_1^2d_2d_3t_1^2 + 12d_1d_2^2d_3t_1^2 +
        36d_2^3d_3t_1^2 + 36d_2^2d_3^2t_1^2 +\\
      &\phantom{{}={}}
        24d_2^2d_3d_4t_1^2 - 48d_1^2d_2d_3t_1t_2 + 24d_1d_2d_3^2t_1t_2 -
        24d_1^2d_3d_4t_1t_2 -\\
      &\phantom{{}={}}
        24d_1d_3^2d_4t_1t_2 -36d_3^3d_4t_1t_2 -36d_3^2d_4^2t_1t_2 +
        216d_1^3d_2d_3 -216d_1^2d_2^2d_3\\
  \hat H_2
      &=4 d_2t_1^4t_2^2t_3^2 - 4d_3t_1^3t_2^3t_3^2 - 4d_2t_1^4t_2^2t_3t_4 +
        8d_2t_1^3t_2^2t_3^2t_4 -4d_1t_1^2t_2^2t_3^2t_4^2 -\\
      &\phantom{{}={}}
        8d_2d_3t_1^4t_2^2 - 16d_3^2t_1^4t_2^2 + 8d_3d_4t_1^4t_2^2 +
        8d_1d_3t_1^3t_2^3 + 16d_3^2t_1^3t_2^3 - 8d_3d_4t_1^3t_2^3 -\\
      &\phantom{{}={}}
        2d_2d_3t_1^4t_2t_3 + 2d_2d_4t_1^4t_2t_3 + 16d_4^2t_1^4t_2t_3 +
        2d_1d_3t_1^3t_2^2t_3 + 30d_2d_3t_1^3t_2^2t_3 -\\
      &\phantom{{}={}}
        8d_1d_4t_1^3t_2^2t_3 - 22d_2d_4t_1^3t_2^2t_3 - 2d_3d_4t_1^3t_2^2t_3 -
        32d_4^2t_1^3t_2^2t_3 - 30d_1d_2t_1^2t_2^2t_3^2 +\\
      &\phantom{{}={}}
        30d_1d_4t_1^2t_2^2t_3^2 + 16d_4^2t_1^2t_2^2t_3^2 - 2d_1d_2t_1^3t_2t_3t_4 +
        2d_2d_3t_1^3t_2t_3t_4 + 30d_1d_2t_1^2t_2^2t_3t_4 -\\
      &\phantom{{}={}}
        30d_1d_3t_1^2t_2^2t_3t_4 + 81d_2^2d_3t_1^4 -54d_2^2d_3t_1^3t_2 -
        162d_2d_3^2t_1^3t_2 - 216d_3^2d_4t_1^3t_2 +\\
      &\phantom{{}={}}
        81d_1^2d_3t_1^2t_2^2 + 405d_1d_3^2t_1^2t_2^2 + 135d_3^2d_4t_1^2t_2^2 -
        162d_2^2d_3t_1^2t_2t_3 - 81d_1^2d_4t_1^2t_2t_3 +\\
      &\phantom{{}={}}
        54d_2^2d_4t_1^2t_2t_3 - 405d_1d_4^2t_1^2t_2t_3 +
        162d_2d_4^2t_1^2t_2t_3 + 162d_1^2d_2t_1t_2t_3t_4 -\\
      &\phantom{{}={}}
        972d_1^2d_2d_3t_1^2 - 324d_1d_2^2d_3t_1^2 - 972d_2^3d_3t_1^2 -
        972d_2^2d_3^2t_1^2 - 648d_2^2d_3d_4t_1^2 +\\
      &\phantom{{}={}}
        1296d_1^2d_2d_3t_1t_2 - 648d_1d_2d_3^2t_1t_2 + 648d_1^2d_3d_4t_1t_2 +
        648d_1d_3^2d_4t_1t_2 +\\
      &\phantom{{}={}}
         972d_3^3d_4t_1t_2 + 972d_3^2d_4^2t_1t_2 - 2916d_1^3d_2d_3 + 2916d_1^2d_2^2d_3.
\end{align*}
\egroup
\nocite{CoCoALib}
\bibliographystyle{alpha}
\bibliography{spectraQVFs}
\end{document}